\documentclass[a4paper,12pt,twoside,leqno,final]{amsart}
\usepackage{amsmath}
\usepackage{amssymb}

\newtheorem{thm}{Theorem}[section]
\newtheorem{lem}[thm]{Lemma}

\newtheorem{prop}[thm]{Proposition}

\theoremstyle{definition}

\newcommand{\C}{{\mathbb C}}
\newcommand{\D}{{\mathbb D}}
\newcommand{\R}{{\mathbb R}}
\newcommand{\T}{{\mathbb T}}
\newcommand{\Z}{{\mathbb Z}}

\newcommand{\La}{\Lambda}

\newcommand{\f}{\frac}
\newcommand{\ov}{\overline}
\newcommand{\al}{\alpha}
\newcommand{\be}{\beta}

\newcommand{\ga}{\gamma}

\newcommand{\la}{\lambda}
\newcommand{\ze}{\zeta}
\renewcommand{\th}{\theta}

\newcommand{\ph}{\varphi}

\newcommand{\const}{\text{\rm const}}

\numberwithin{equation}{section}

\title[Functions with small and large spectra]
{Functions with small and large spectra\\ 
as (non)extreme points in subspaces of $H^\infty$}
\dedicatory{Dedicated to Nikolai Kapitonovich Nikolski on the occasion of his 80th birthday}
\author{Konstantin M. Dyakonov}
\address{Departament de Matem\`atiques i Inform\`atica, IMUB, BGSMath, Universitat de Barcelona, Gran Via 585, E-08007 Barcelona, Spain}
\address{ICREA, Pg. Llu\'is Companys 23, E-08010 Barcelona, Spain}
\email{konstantin.dyakonov@icrea.cat}
\keywords{Bounded analytic functions, spectral gaps, lacunary polynomials, extreme points}
\subjclass[2010]{30C10, 30H10, 42A05, 46A55}
\thanks{Supported in part by grant MTM2017-83499-P from El Ministerio de Ciencia e Innovaci\'on (Spain) and grant 2017-SGR-358 from AGAUR (Generalitat de Catalunya).}

\begin{document}
\begin{abstract}
Given a subset $\Lambda$ of $\mathbb Z_+:=\{0,1,2,\dots\}$, let $H^\infty(\Lambda)$ denote the space of bounded analytic functions $f$ on the unit disk whose coefficients $\widehat f(k)$ vanish for $k\notin\Lambda$. Assuming that either $\Lambda$ or $\mathbb Z_+\setminus\Lambda$ is finite, we determine the extreme points of the unit ball in $H^\infty(\Lambda)$.
\end{abstract}

\maketitle

\section{Introduction}

Let $H^\infty$ stand for the space of bounded holomorphic functions on the disk $\D:=\{z\in\C:|z|<1\}$. As usual, a function $f\in H^\infty$ is identified with its boundary trace on the circle $\T:=\partial\D$, defined almost everywhere in the sense of nontangential convergence. We thus embed $H^\infty$ in $L^\infty=L^\infty(\T)$, the space of essentially bounded functions on $\T$, bearing in mind that the quantity 
$$\|f\|_\infty:=\sup\{|f(z)|:\,z\in\D\}$$
agrees, for $f\in H^\infty$, with the $L^\infty$ norm of the boundary function $f\big|_\T$. The underlying theory and other basic facts about $H^\infty$ can be found in any of \cite{G, Hof, K}. 

\par We shall be concerned with the geometry of the unit ball---specifically, with the structure of its extreme points---in certain subspaces of $H^\infty$. These will appear shortly, once a bit of terminology and notation is fixed. 

\par Given a (complex) Banach space $X=(X,\|\cdot\|)$, we write  
$$\text{\rm ball}(X):=\{x\in X:\,\|x\|\le1\}$$
for the closed unit ball of $X$. Also, we recall that a point in $\text{\rm ball}(X)$ is said to be {\it extreme} for the ball if it is not the midpoint of any two distinct points in $\text{\rm ball}(X)$. 

\par Further, with an integrable function $f$ on $\T$ we associate the sequence of its {\it Fourier coefficients}
$$\widehat f(k):=\f 1{2\pi}\int_\T\ov\ze^kf(\ze)\,|d\ze|,\qquad k\in\Z,$$
and the set 
$$\text{\rm spec}\,f:=\{k\in\Z:\,\widehat f(k)\ne0\},$$
known as the {\it spectrum} of $f$. Thus, in particular, 
$$H^\infty=\{f\in L^\infty:\,\text{\rm spec}\,f\subset\Z_+\},$$
where $\Z_+$ stands for the set of nonnegative integers. 

\par The geometry of the unit ball in $H^\infty$, let alone $L^\infty$, seems to be well understood. To begin with, it is worth mentioning that the extreme points of $\text{\rm ball}(L^\infty)$ are precisely the unimodular functions on $\T$. As to $\text{\rm ball}(H^\infty)$, its extreme points are characterized among the unit-norm functions $f\in H^\infty$ by the weaker condition that 
\begin{equation}\label{eqn:logintdiv}
\int_\T\log(1-|f(\ze)|)\,|d\ze|=-\infty
\end{equation}
(see, e.g., \cite[Section V]{dLR} or \cite[Chapter 9]{Hof}). 

\par Our purpose here is to see what happens for subspaces of $H^\infty$ that are formed by functions with prescribed spectral gaps. Precisely speaking, given a subset $\Lambda$ of $\mathbb Z_+$, we consider the space
$$H^\infty(\Lambda):=\{f\in H^\infty:\,\text{\rm spec}\,f\subset\La\},$$
with norm $\|\cdot\|_\infty$, and we seek to characterize the extreme points of $\text{\rm ball}(H^\infty(\Lambda))$. This will be accomplished in two special cases that represent two \lq\lq extreme" situations. Namely, it will be assumed that either $\La$ or $\Z_+\setminus\La$ is a finite set (this dichotomy accounts for the phrase \lq\lq small and large spectra" in the paper's title). The results pertaining to each of these cases will be stated in Section 2 below, and then proved in Sections 3 and 4. 

\par Meanwhile, we mention that similar questions have already been studied in the context of the Hardy space $H^1$. The extreme points of $\text{\rm ball}(H^1)$ were identified by de Leeuw and Rudin \cite{dLR} as outer functions of norm $1$. The case of 
$$H^1(\La):=\{f\in H^1:\,\text{\rm spec}\,f\subset\La\}$$
was recently settled by the author for sets $\La\subset\Z_+$ that are either finite (see \cite{DAdv2021}) or have finite complement in $\Z_+$ (see \cite{DCRM, DARX}). Among the finite $\La$'s, we single out the \lq\lq gapless" sets of the form
\begin{equation}\label{eqn:lan}
\La_N:=\{0,1,\dots,N\},
\end{equation}
with $N$ a positive integer, in which case we are dealing with the space of polynomials of degree at most $N$. For this last space, endowed with the $L^1$ norm over $\T$, the extreme points of the unit ball were described earlier in \cite{DMRL2000}; alternatively, the description follows from \cite[Theorem 6]{DPAMS}. 

\par Going back to the $H^\infty(\La)$ setting, we remark that the nonlacunary polynomial case, where $\La=\La_N$, was treated previously in \cite{DMRL2003}. When moving to general finite sets $\La$, however, we have to face new complications. For spaces of trinomials, which arise when $\#\La=3$, a detailed analysis was carried out by Neuwirth in \cite{Neu}; there, both the extreme and exposed points of the unit ball were determined. (By definition, given a Banach space $X$, a point $x\in\text{\rm ball}(X)$ is {\it exposed} for the ball if there exists a functional $\phi\in X^*$ of norm $1$ such that the set $\{y\in\text{\rm ball}(X):\phi(y)=1\}$ equals $\{x\}$.) On the other hand, a theorem of Amar and Lederer (see \cite{AL}) tells us that the exposed points of $\text{\rm ball}(H^\infty)$ are precisely the unit-norm functions $f\in H^\infty$ for which the set $\{\ze\in\T:|f(\ze)|=1\}$ has positive measure. 

\par Here, we make no attempt to characterize the exposed points of $\text{\rm ball}(H^\infty(\La))$. Rather, we mention this as an open problem. When $\#\La<\infty$ or $\#(\Z_+\setminus\La)<\infty$, one might probably arrive at a solution with relatively light machinery, via a suitable adaptation of our current techniques. 

\par Restricting our attention to the extreme points of $\text{\rm ball}(H^\infty(\La))$, as we do here, we are still puzzled by the case where both $\La$ and $\Z_+\setminus\La$ are infinite sets. It would be nice to gain some understanding of what happens for such $\La$'s. In particular, we wonder which arithmetic properties of $\La$ (if any) are relevant to the problem. A more specific question related to condition \eqref{eqn:logintdiv} is raised in Section 2 below, next to Theorem \ref{thm:fincodim}. 

\par Finally, we mention yet another type of subspaces in $H^\infty$ where the structure of the extreme points remains unclear. Namely, given an inner function $\th$, we consider the {\it model subspace} $K_\th^\infty:=H^\infty\cap\th\ov z\ov{H^\infty}$ and ask for a characterization of the extreme points of $\text{\rm ball}(K_\th^\infty)$. This problem was originally posed in \cite{DIEOT}; see also \cite{DPAMS} for a treatment of its $L^1$ counterpart, where a simple solution is available. Except for the case of $\th(z)=z^{N+1}$, when $K_\th^\infty$ agrees with $H^\infty(\La_N)$, the two types of spaces (i.e., $H^\infty(\La)$ and $K_\th^\infty$) are rather different in nature, though.

\section{Statement of results}

\par We begin with the case where $\Z_+\setminus\La$ is finite, since a neater formulation is then available and the result is easier to establish. In fact, the extreme point criterion that arises in this case for $H^\infty(\La)$ turns out to be the same as for $H^\infty$. 

\begin{thm}\label{thm:fincodim} Let $\La\subset\Z_+$ be a set with
\begin{equation}\label{eqn:complfin}
\#(\Z_+\setminus\La)<\infty.
\end{equation}
Suppose further that $f\in H^\infty(\La)$ and $\|f\|_\infty=1$. Then $f$ is an extreme point of $\text{\rm ball}(H^\infty(\La))$ if and only if it satisfies \eqref{eqn:logintdiv}.
\end{thm}

\par It would be interesting to find a complete description of the sets $\La\subset\Z_+$ with the property that the extreme points of $\text{\rm ball}(H^\infty(\La))$ are characterized by \eqref{eqn:logintdiv}. One feels that such $\La$'s should be suitably \lq\lq thick" in $\Z_+$, but the sufficient condition \eqref{eqn:complfin} is certainly far from being necessary. It seems plausible that an appropriate sparseness condition on $\Z_+\setminus\La$ would actually suffice. At the same time, for a set $\La$ with the desired property, it may well happen that $\Z_+\setminus\La$ is no thinner (in whatever sense) than $\La$ itself, as we shall now see. 

\par By way of example, take $\La$ to be $2\Z_+$, the set of nonnegative even integers. Now let $f\in H^\infty(2\Z_+)$ be a function with $\|f\|_\infty=1$. Assuming that 
$$\int_\T\log(1-|f(\ze)|)\,|d\ze|>-\infty,$$
we put
\begin{equation}\label{eqn:outfun}
g(z):=\exp\left\{\f1{2\pi}\int_\T\f{\ze+z}{\ze-z}\log(1-|f(\ze)|)\,|d\ze|\right\},
\qquad z\in\D,
\end{equation}
so that $g$ is the outer function with modulus $1-|f|$ on $\T$. Furthermore, $g$ is an even function in $H^\infty$ (because $f$ is even) and hence $g\in H^\infty(2\Z_+)$. Also, $\|f\pm g\|_\infty\le1$. Thus, $f+g$ and $f-g$ are two distinct points of $\text{\rm ball}(H^\infty(2\Z_+))$, while $f$ is their midpoint. This proves the necessity of \eqref{eqn:logintdiv} in order that $f$ be an extreme point of $\text{\rm ball}(H^\infty(2\Z_+))$. The sufficiency is trivial, since $H^\infty(2\Z_+)\subset H^\infty$. 

\par We now mention an analogue of Theorem \ref{thm:fincodim} where the underlying space is taken to be the {\it disk algebra} $C_A:=H^\infty\cap C(\T)$ instead of $H^\infty$. This time, $H^\infty(\La)$ gets replaced by 
$$C_A(\La):=H^\infty(\La)\cap C_A$$ 
and we have the following result. 

\begin{prop}\label{prop:fincodimca} Given a set $\La\subset\Z_+$ satisfying \eqref{eqn:complfin}, the extreme points of $\text{\rm ball}(C_A(\La))$ are precisely the unit-norm functions $f\in C_A(\La)$ with property \eqref{eqn:logintdiv}.
\end{prop}

\par Next, we turn to the case where $\La$ is a finite subset of $\Z_+$. The $H^\infty$ functions with spectrum in $\La$ are now polynomials of the form 
$$p(z)=\sum_{k\in\La}\widehat p(k)z^k,$$
and we prefer to denote the set of such polynomials by $\mathcal P(\La)$ rather than by $H^\infty(\La)$. Of course, $\mathcal P(\La)$ is still endowed with the supremum norm, and we shall occasionally write $\mathcal P^\infty(\La)$ for the normed space $(\mathcal P(\La),\|\cdot\|_\infty)$ that arises. 

\par We shall henceforth assume (without losing anything of substance) that $0\in\La$ and 
$\#\La\ge2$, so that 
\begin{equation}\label{eqn:deflambda}
\La=\{0,1,\dots,N\}\setminus\{k_1,\dots,k_M\}
\end{equation}
for some positive integers $N$ and $k_j$ ($j=1,\dots,M$) with 
$$k_1<k_2<\dots<k_M<N.$$
In the special case where $M=0$, the set $\{k_1,\dots,k_M\}$ is empty, so $\La$ becomes $\La_N$ (as defined by \eqref{eqn:lan}) and $\mathcal P(\La)$ reduces to 
\begin{equation}\label{eqn:pollan}
\mathcal P_N:=\mathcal P(\La_N),
\end{equation}
the space of polynomials of degree at most $N$. In this nonlacunary case, the extreme points of the unit ball were previously characterized in \cite{DMRL2003}. Here, we refine the method of \cite{DMRL2003} to deal with spaces of {\it lacunary polynomials} (or {\it fewnomials}) that arise as $\mathcal P(\La)$ for general sets $\La$ of the form \eqref{eqn:deflambda}. 

\par Among the unit-norm polynomials in $\mathcal P^\infty(\La)$, the simplest examples are provided by the {\it monomials} $z\mapsto cz^k$, with $k\in\La$ and $c$ a unimodular constant. Clearly, any such monomial is an extreme point of $\text{\rm ball}(\mathcal P^\infty(\La))$, so we may exclude these \lq\lq trivial" extreme points from further consideration. 

\par Now suppose $p\in\mathcal P(\La)$ is a polynomial with $\|p\|_\infty=1$ whose spectrum contains at least two elements. Our criterion for $p$ to be extreme in $\text{\rm ball}(\mathcal P^\infty(\La))$ will be stated in terms of a certain matrix $\mathcal M=\mathcal M_\La(p)$ associated with $p$, and we proceed with the construction of $\mathcal M$. 

\par Let $\ze_1,\dots,\ze_n$ be an enumeration of the (finite and nonempty) set $\{\ze\in\T:|p(\ze)|=1\}$. Viewed as zeros of the function 
$$\tau(z):=1-|p(z)|^2,\qquad z\in\T$$
(or equivalently, of the polynomial $z^N\tau$), the $\ze_j$'s have even multiplicities, which we denote by $2\mu_1,\dots,2\mu_n$ respectively; the $\mu_j$'s are therefore positive integers. We then put
\begin{equation}\label{eqn:mugadef}
\mu:=\sum_{j=1}^n\mu_j\quad\text{\rm and}\quad\ga:=\mu/2.
\end{equation}
Since $z^N\tau\in\mathcal P_{2N}$, it follows that $\mu\le N$. 

\par For each $j\in\{1,\dots,n\}$, we consider the Wronski-type matrix 
$$W_j:=
\begin{pmatrix}
\ov\ze_j^\ga p(\ze_j) & \ov\ze_j^{\ga+1}p(\ze_j) & \dots & \ov\ze_j^{N-\ga}p(\ze_j)\\
\left(\ov z^\ga p\right)'(\ze_j) & \left(\ov z^{\ga+1}p\right)'(\ze_j) & \dots 
& \left(\ov z^{N-\ga}p\right)'(\ze_j)\\
\vdots & \vdots & \vdots & \vdots\\
\left(\ov z^\ga p\right)^{(\mu_j-1)}(\ze_j) & \left(\ov z^{\ga+1}p\right)^{(\mu_j-1)}(\ze_j) & \dots & \left(\ov z^{N-\ga}p\right)^{(\mu_j-1)}(\ze_j)
\end{pmatrix},
$$
which has $\mu_j$ rows and $N-\mu+1$ columns (indeed, the exponent $N-\ga$ in the last column equals $\ga+N-\mu$). Here, the convention is that the independent variable $z=e^{it}$ lives on $\T$ and that differentiation is with respect to the real parameter $t=\arg z$. More precisely, expressions of the form $(\ov z^{\ga+\ell}p)^{(s)}(\ze_j)$ with $\ell,s\in\Z_+$ should be interpreted as 
$$\f{d^s}{dt^s}\left\{e^{-i(\ga+\ell)t}p(e^{it})\right\}\big|_{t=t_j},$$
where $t_j\in(-\pi,\pi]$ is defined by $e^{it_j}=\ze_j$. We also need the real matrices 
$$\mathcal U_j:=\text{\rm Re}\,W_j\quad\text{\rm and}\quad\mathcal V_j:=\text{\rm Im}\,W_j
\qquad(j=1,\dots n).$$

\par The rest of the construction involves the polynomial 
\begin{equation}\label{eqn:polyrz}
r(z):=\prod_{j=1}^n(z-\ze_j)^{\mu_j}
\end{equation}
and its coefficients $\widehat r(k)$ with $k\in\Z$. (For $k<0$ and $k>\mu$, we obviously have $\widehat r(k)=0$.) From these, some further matrices will be built. Namely, we introduce the $M\times(N-\mu+1)$ matrix 
$$\mathcal R:=
\begin{pmatrix}
\widehat r(k_1) & \widehat r(k_1-1) & \dots & \widehat r(k_1-N+\mu)\\
\vdots & \vdots & \vdots & \vdots\\
\widehat r(k_M) & \widehat r(k_M-1) & \dots & \widehat r(k_M-N+\mu)
\end{pmatrix}
$$
along with the real matrices 
$$\mathcal A:=\text{\rm Re}\,\mathcal R\,\quad\text{\rm and}\quad
\mathcal B:=\text{\rm Im}\,\mathcal R.$$
Finally, we define the block matrix 
\begin{equation}\label{eqn:blockma}
\mathcal M=\mathcal M_\La(p):=
\begin{pmatrix}
\mathcal A & -\mathcal B\\
\mathcal B & \mathcal A\\
\mathcal U_1 & \mathcal V_1\\
\vdots & \vdots\\
\mathcal U_n & \mathcal V_n
\end{pmatrix},
\end{equation}
which has $2M+\mu$ rows and $2(N-\mu+1)$ columns. 

\begin{thm}\label{thm:findim} Given a set $\La\subset\Z_+$ of the form \eqref{eqn:deflambda}, 
suppose that $p$ is a unit-norm polynomial in $\mathcal P^\infty(\La)$ distinct from a monomial. Then $p$ is an extreme point of $\text{\rm ball}(\mathcal P^\infty(\La))$ if and only if $\text{\rm rank}\,\mathcal M_\La(p)=2(N-\mu+1)$.
\end{thm}

\par Even though the rank condition above may appear somewhat bizarre, it is unlikely that the criterion could be substantially simplified. In fact, even in the nonlacunary polynomial space \eqref{eqn:pollan}, and already for $N=2$, one can find unit-norm polynomials $p_1$, $p_2$ satisfying 
$$1-|p_1(z)|^2=2\left(1-|p_2(z)|^2\right),\qquad z\in\T,$$
and such that $p_1$ is a non-extreme point of the unit ball, while $p_2$ is extreme; see \cite[p.\,720]{DMRL2003} for an example. This means that, even for $\mathcal P_2$, the extreme point criterion cannot be stated in terms of the $\ze_j$'s and $\mu_j$'s alone, so a certain level of complexity seems to be unavoidable.

\section{Proofs of Theorem \ref{thm:fincodim} and Proposition \ref{prop:fincodimca}}

\noindent{\it Proof of Theorem \ref{thm:fincodim}.} Let $f\in H^\infty(\La)$ and $\|f\|_\infty=1$. Assuming \eqref{eqn:logintdiv}, we know that $f$ is an extreme point of $\text{\rm ball}(H^\infty)$ and hence also of the smaller set $\text{\rm ball}(H^\infty(\La))$. 
\par Conversely, assume that \eqref{eqn:logintdiv} fails, so that 
\begin{equation}\label{eqn:logintconv}
\int_\T\log(1-|f(\ze)|)\,|d\ze|>-\infty.
\end{equation}
Then we can find a function $g\in H^\infty$, $g\not\equiv0$, satisfying
\begin{equation}\label{eqn:gdomf}
|g|\le1-|f|
\end{equation}
almost everywhere on $\T$ (e.g., take $g$ to be the outer function with modulus $1-|f|$, as defined by \eqref{eqn:outfun}). Further, letting
$$m:=\#(\Z_+\setminus\La)$$
and recalling the notation $\mathcal P_m$ for the set of polynomials of degree at most $m$, we go on to claim that there exists $p_0\in\mathcal P_m$, $p_0\not\equiv0$, for which $gp_0\in H^\infty(\La)$. To see why, write 
$$\Z_+\setminus\La=\{k_1,\dots,k_m\},$$
where $k_1,\dots,k_m$ are pairwise distinct integers, and consider the linear operator $T:\mathcal P_m\to\C^m$ that acts by the rule
$$Tp:=\left(\widehat{(gp)}(k_1),\dots,\widehat{(gp)}(k_m)\right),\qquad p\in\mathcal P_m.$$
Because $\text{\rm dim}\,\mathcal P_m=m+1$, while the rank of $T$ does not exceed $m$, the rank-nullity theorem (see, e.g., \cite[p.\,63]{Axl}) tells us that $\text{\rm Ker}\,T$, the null-space of $T$, has dimension at least $1$ and is therefore nontrivial. 
\par Now, if $p_0$ is any non-null polynomial in $\text{\rm Ker}\,T$, then 
\begin{equation}\label{eqn:hatgp}
\widehat{(gp_0)}(k_1)=\dots=\widehat{(gp_0)}(k_m)=0,
\end{equation}
and so $gp_0$ is a nontrivial function in $H^\infty(\La)$. We may also assume that $\|p_0\|_\infty\le1$, and together with \eqref{eqn:gdomf} this yields
\begin{equation}\label{eqn:pmineq}
|f\pm gp_0|\le|f|+|g||p_0|\le|f|+|g|\le1
\end{equation}
almost everywhere on $\T$. Consequently, 
$$f\pm gp_0\in\text{\rm ball}(H^\infty(\La))$$ 
and the identity 
\begin{equation}\label{eqn:idfgp}
f=\f12(f+gp_0)+\f12(f-gp_0)
\end{equation}
shows that $f$ is not an extreme point of $\text{\rm ball}(H^\infty(\La))$. \qed

\bigskip\noindent{\it Proof of Proposition \ref{prop:fincodimca}.} Once again, we only have to check that every unit-norm function $f\in C_A(\La)$ satisfying \eqref{eqn:logintconv} is non-extreme in $\text{\rm ball}(C_A(\La))$. 
\par For any such $f$ (and actually for any $f\in C_A$ with $\|f\|_\infty\le1$), condition \eqref{eqn:logintconv} enables us to find a non-null function $g\in C_A$ that obeys \eqref{eqn:gdomf}; see \cite[Chapter 9]{Hof}. Now, using this $g$ in place of its namesake above, while keeping the rest of notation, we can readily adjust the preceding proof to the current situation. Namely, we construct (exactly as before) a polynomial $p_0\in\mathcal P_m$ with $0<\|p_0\|_\infty\le1$ that makes \eqref{eqn:hatgp} true. The product $gp_0$ is then a nontrivial function in $C_A(\La)$, and since \eqref{eqn:pmineq} is again valid, it follows that 
$$f\pm gp_0\in\text{\rm ball}(C_A(\La)).$$
Finally, we infer from \eqref{eqn:idfgp} that $f$ is a non-extreme point of $\text{\rm ball}(C_A(\La))$. \qed

\section{Proof of Theorem \ref{thm:findim}}

We begin by stating and proving a preliminary result. 

\begin{lem}\label{thm:twoconst} Given a finite set $\La\subset\Z_+$, suppose that $p\in\mathcal P(\La)$ and $\|p\|_\infty=1$. The following conditions are equivalent: 
\par{\rm(i)} $p$ is not an extreme point of $\text{\rm ball}(\mathcal P^\infty(\La))$. 
\par{\rm(ii)} There exist positive constants $C_1$, $C_2$ and a non-null polynomial $q\in\mathcal P(\La)$ such that 
\begin{equation}\label{eqn:fircon}
|q|^2\le C_1\left(1-|p|^2\right)
\end{equation}
and 
\begin{equation}\label{eqn:seccon}
|\text{\rm Re}(\ov pq)|\le C_2\left(1-|p|^2\right)
\end{equation}
everywhere on $\T$. 
\end{lem}

\begin{proof} Clearly, (i) holds if and only if there exists a non-null polynomial $q\in\mathcal P(\La)$ for which 
\begin{equation}\label{eqn:ppmqleone}
\|p+q\|_\infty\le1\quad\text{\rm and}\quad\|p-q\|_\infty\le1.
\end{equation}
An obvious restatement of \eqref{eqn:ppmqleone} is that $|p\pm q|^2\le1$ on $\T$; and since 
$$|p\pm q|^2=|p|^2\pm2\text{\rm Re}(\ov pq)+|q|^2,$$
while $\max(a,-a)=|a|$ for all $a\in\R$, we may further rewrite \eqref{eqn:ppmqleone} in the form 
\begin{equation}\label{eqn:ppmqleonebis}
2\left|\text{\rm Re}(\ov pq)\right|+|q|^2\le1-|p|^2.
\end{equation}

\par Now, if \eqref{eqn:ppmqleonebis} is fulfilled for some nontrivial $q\in\mathcal P(\La)$, then \eqref{eqn:fircon} and \eqref{eqn:seccon} are sure to hold (for the same $q$) with $C_1=1$ and $C_2=\f12$.

\par Conversely, suppose $q\in\mathcal P(\La)$ is a nontrivial polynomial that satisfies \eqref{eqn:fircon} and \eqref{eqn:seccon}. Replacing $q$ by $\varepsilon q$ with a suitable $\varepsilon>0$ if necessary, we can arrange it for $C_1$ and $C_2$ to be as small as desired. In particular, we may assume that $C_1\le\f12$ and $C_2\le\f14$. The resulting inequalities 
$$|q|^2\le\f12\left(1-|p|^2\right)$$ 
and 
$$2|\text{\rm Re}(\ov pq)|\le\f12\left(1-|p|^2\right)$$
imply \eqref{eqn:ppmqleonebis} and hence \eqref{eqn:ppmqleone}. 
\end{proof}

\medskip\noindent{\it Proof of Theorem \ref{thm:findim}.} Suppose that $p$ satisfies the hypotheses of the theorem and fails to be an extreme point of $\text{\rm ball}(\mathcal P^\infty(\La))$. By Lemma \ref{thm:twoconst}, we can find a polynomial $q\in\mathcal P(\La)$, $q\not\equiv0$, that makes \eqref{eqn:fircon} and \eqref{eqn:seccon} true for some constants $C_1,C_2>0$. 

\par Now, for each $j\in\{1,\dots,n\}$, we have 
$$1-|p(z)|^2=O\left(|z-\ze_j|^{2\mu_j}\right)$$
as $z(\in\T)$ tends to $\ze_j$. In conjunction with \eqref{eqn:fircon}, this yields 
$$|q(z)|^2=O\left(|z-\ze_j|^{2\mu_j}\right),$$
or equivalently, 
$$|q(z)|=O\left(|z-\ze_j|^{\mu_j}\right)$$
near $\ze_j$. Thus, $q$ has a zero of multiplicity at least $\mu_j$ at $\ze_j$. It follows that $q$ is divisible by the polynomial $r$ given by \eqref{eqn:polyrz}; and since $q\in\mathcal P_N$, while $r\in\mathcal P_\mu$, we see that
\begin{equation}\label{eqn:factofq}
q=q_0r
\end{equation}
for some (non-null) $q_0\in\mathcal P_{N-\mu}$. 

\par Our next step is to exploit \eqref{eqn:seccon}, so as to gain further information about $q_0$. But first we need to derive a more convenient expression for $r$. Given $j\in\{1,\dots, n\}$, we write $\ze_j=e^{it_j}$ and note that, for $z=e^{it}\in\T$, we have the identity 
\begin{equation}\label{eqn:trigiden}
z-\ze_j=e^{it/2}e^{it_j/2}\cdot2i\sin\f{t-t_j}2.
\end{equation}
Here and throughout, it is assumed that 
\begin{equation}\label{eqn:argarg}
t:=\arg z\quad\text{\rm and}\quad t_j:=\arg\ze_j,
\end{equation}
where \lq\lq arg" stands for the principal branch of the argument (i.e., the one with values in $(-\pi,\pi]$). In particular, $t$ (resp., $t_j$) is uniquely determined by $z$ (resp., $\ze_j$), and we put 
$$\ph_j(z):=2\sin\f{t-t_j}2,\qquad z\in\T.$$
Clearly, $\ph_j$ is real-valued and 
\begin{equation}\label{eqn:phidist}
|\ph_j(z)|=|z-\ze_j|,\qquad z\in\T,
\end{equation}
this last property being immediate from \eqref{eqn:trigiden}. 
We then rewrite \eqref{eqn:trigiden} in the form 
\begin{equation}\label{eqn:trigidenbis}
z-\ze_j=iz^{1/2}\ze_j^{1/2}\ph_j(z)
\end{equation}
(with the appropriate determination of the square root). Raising both sides of \eqref{eqn:trigidenbis} to the power $\mu_j$ and taking products yields 
\begin{equation}\label{eqn:newr}
r(z)=\la z^{\ga}\prod_{j=1}^n\left(\ph_j(z)\right)^{\mu_j},\qquad z\in\T,
\end{equation}
where
$$\la:=i^\mu\prod_{j=1}^n\ze_j^{\mu_j/2}$$
and $\ga:=\mu/2$, in accordance with \eqref{eqn:mugadef}. We note that $\la$ is a unimodular constant depending only on the $\ze_j$'s and $\mu_j$'s. 

\par Further, we combine \eqref{eqn:factofq} and \eqref{eqn:newr} to get
$$\text{\rm Re}\left(\ov{p(z)}q(z)\right)=\prod_{j=1}^n\left(\ph_j(z)\right)^{\mu_j}
\text{\rm Re}\left(\la z^{\ga}\ov{p(z)}q_0(z)\right),\qquad z\in\T.$$
In view of \eqref{eqn:phidist}, this implies that
\begin{equation}\label{eqn:repq}
\left|\text{\rm Re}\left(\ov{p(z)}q(z)\right)\right|=\prod_{j=1}^n|z-\ze_j|^{\mu_j}
\left|\text{\rm Re}\left(\la z^{\ga}\ov{p(z)}q_0(z)\right)\right|.
\end{equation}
On the other hand, 
\begin{equation}\label{eqn:tauasymp}
1-|p(z)|^2\asymp\prod_{j=1}^n|z-\ze_j|^{2\mu_j},\qquad z\in\T.
\end{equation}
(As usual, the sign $\asymp$ means that the ratio of the two quantities stays in the interval $[C^{-1},C]$ for some constant $C>1$.) Taking \eqref{eqn:repq} and \eqref{eqn:tauasymp} into account, we now rewrite \eqref{eqn:seccon} as 
\begin{equation}\label{eqn:shvabra}
\left|\text{\rm Re}\left(\la z^{\ga}\ov{p(z)}q_0(z)\right)\right|
\le\const\cdot\prod_{j=1}^n|z-\ze_j|^{\mu_j},\qquad z\in\T.
\end{equation}
Thus, for every $j\in\{1,\dots,n\}$, the function 
$$z\left(=e^{it}\right)\mapsto\text{\rm Re}\left(\la z^{\ga}\ov{p(z)}q_0(z)\right)$$
has a zero of multiplicity at least $\mu_j$ at $\ze_j$. This fact admits an obvious restatement in terms of derivatives; namely, for $j=1,\dots,n$ we have
\begin{equation}\label{eqn:dereqzero}
\text{\rm Re}\left(\la z^{\ga}\ov{p(z)}q_0(z)\right)^{(s)}(\ze_j)=0,\qquad s=0,\dots,\mu_j-1.
\end{equation}
(To keep on the safe side, we recall \eqref{eqn:argarg} and emphasize that the derivatives in \eqref{eqn:dereqzero} are actually taken with respect to $t$ and computed at $t_j$. In particular, differentiation commutes with the real part operator.) Now, we write the polynomial $q_0$ in the form 
\begin{equation}\label{eqn:qalbe}
q_0(z)=\ov\la\sum_{l=0}^{N-\mu}(\al_l+i\be_l)z^l,
\end{equation}
where $\al_l$ and $\be_l$ are real parameters, and plug this expression into \eqref{eqn:dereqzero}. This done, we obtain for each $j\in\{1,\dots,n\}$ the $\mu_j$ equations
$$\sum_{l=0}^{N-\mu}\al_l\,\text{\rm Re}\left(z^{\ga+l}\ov p\right)^{(s)}(\ze_j)
-\sum_{l=0}^{N-\mu}\be_l\,\text{\rm Im}\left(z^{\ga+l}\ov p\right)^{(s)}(\ze_j)=0,$$
or equivalently, 
\begin{equation}\label{eqn:lineq}
\sum_{l=0}^{N-\mu}\text{\rm Re}\left(\ov z^{\ga+l}p\right)^{(s)}(\ze_j)\cdot\al_l
+\sum_{l=0}^{N-\mu}\text{\rm Im}\left(\ov z^{\ga+l}p\right)^{(s)}(\ze_j)\cdot\be_l=0,
\end{equation}
with $s=0,\dots,\mu_j-1$. We have thus a total of $\mu_1+\dots+\mu_n=\mu$ equations here. 

\par Furthermore, we want to recast the condition that $q\in\mathcal P(\La)$ in terms of our $\al_l$'s and $\be_l$'s. Since $\La$ is given by \eqref{eqn:deflambda}, we know that
\begin{equation}\label{eqn:hatqknu}
\widehat q(k_\nu)=0\quad\text{\rm for}\quad\nu=1,\dots,M.
\end{equation}
On the other hand, setting 
$$A_k:=\text{\rm Re}\,\widehat r(k)\quad\text{\rm and}\quad B_k:=\text{\rm Im}\,\widehat r(k),\qquad k\in\Z,$$
we use \eqref{eqn:factofq} and \eqref{eqn:qalbe} to find that 
\begin{equation*}
\begin{aligned}
\widehat q(k_\nu)&=\sum_{l=0}^{N-\mu}\widehat q_0(l)\widehat r(k_\nu-l)\\
&=\ov\la\sum_{l=0}^{N-\mu}(\al_l+i\be_l)\left(A_{k_\nu-l}+iB_{k_\nu-l}\right)
\end{aligned}
\end{equation*}
for each $\nu$. Consequently, \eqref{eqn:hatqknu} can be rephrased by saying that
\begin{equation}\label{eqn:repaeq}
\sum_{l=0}^{N-\mu}\left(A_{k_\nu-l}\,\al_l-B_{k_\nu-l}\,\be_l\right)=0,\qquad\nu=1,\dots,M,
\end{equation}
and
\begin{equation}\label{eqn:impaeq}
\sum_{l=0}^{N-\mu}\left(B_{k_\nu-l}\,\al_l+A_{k_\nu-l}\,\be_l\right)=0,\qquad\nu=1,\dots,M.
\end{equation}

\par Taken together, the $2M+\mu$ equations that appear above as \eqref{eqn:lineq},  \eqref{eqn:repaeq}, and \eqref{eqn:impaeq} tell us that the vector 
\begin{equation}\label{eqn:coevec}
(\al_0,\dots,\al_{N-\mu},\be_0,\dots,\be_{N-\mu})
\end{equation}
belongs to $\text{\rm Ker}\,\mathcal M$, the kernel of the linear map 
$$\mathcal M:\,\R^{2(N-\mu+1)}\to\R^{2M+\mu}$$
given by \eqref{eqn:blockma}. The polynomial $q$ (and hence $q_0$) being non-null, we see that the vector \eqref{eqn:coevec} is nonzero, and so 
\begin{equation}\label{eqn:kernontriv}
\text{\rm Ker}\,\mathcal M\ne\{0\}.
\end{equation}
Now, because 
$$\text{\rm dim}(\text{\rm Ker}\,\mathcal M)+\text{\rm rank}\,\mathcal M=2(N-\mu+1)$$
by virtue of the rank-nullity theorem (see \cite[p.\,63]{Axl}), we may further restate \eqref{eqn:kernontriv} in the form
\begin{equation}\label{eqn:rankless}
\text{\rm rank}\,\mathcal M<2(N-\mu+1).
\end{equation}
To summarize, we have proved that if $p$ is a non-extreme point of $\text{\rm ball}(\mathcal P^\infty(\La))$, then \eqref{eqn:rankless} holds. 

\par The converse is actually true as well, since every step in the above reasoning can be reversed. Indeed, assuming \eqref{eqn:rankless}, we rewrite it as \eqref{eqn:kernontriv} and take \eqref{eqn:coevec} to be any nonzero vector in $\text{\rm Ker}\,\mathcal M$. Then we define the polynomials $q_0$ and $q$, in this order, by means of \eqref{eqn:qalbe} and \eqref{eqn:factofq}. Equations \eqref{eqn:repaeq} and \eqref{eqn:impaeq} yield \eqref{eqn:hatqknu}, and it follows that $q$ is a non-null polynomial in $\mathcal P(\La)$. Moreover, conditions \eqref{eqn:fircon} and \eqref{eqn:seccon} are then fulfilled. In fact, \eqref{eqn:fircon} is immediate from \eqref{eqn:factofq} and \eqref{eqn:tauasymp}, while \eqref{eqn:seccon} is ensured by \eqref{eqn:lineq}. (One should recall that \eqref{eqn:lineq} is expressible as \eqref{eqn:dereqzero} and implies \eqref{eqn:shvabra}, which is equivalent to \eqref{eqn:seccon}.) Finally, we invoke Lemma \ref{thm:twoconst} to conclude that $p$ is not an extreme point of $\text{\rm ball}(\mathcal P^\infty(\La))$. 

\par Now we know that a unit-norm polynomial $p\in\mathcal P^\infty(\La)$ is a non-extreme point of the unit ball if and only if the associated matrix $\mathcal M=\mathcal M_\La(p)$ satisfies \eqref{eqn:rankless}. In other words, the extreme points---other than monomials---are characterized by the condition 
$$\text{\rm rank}\,\mathcal M=2(N-\mu+1).$$
The proof is complete. 
\qed

\end{document}